\documentclass[11pt]{amsart}
\usepackage[english]{babel}

\usepackage{xcolor}
\usepackage{mathtools}
\usepackage{amssymb}
\usepackage{amsfonts}
\usepackage{amstext}
\usepackage{amsmath}
\usepackage{amsthm}
\newtheorem{theorem}{Theorem}[section]
\newtheorem{Definition}{Definition}[section]
\newtheorem{Corollary}{Corollary}[section]
\newtheorem{prop}{Proposition}[section]
\newtheorem{lemma}[theorem]{Lemma}

\newcommand{\m}{\mu}
\newcommand{\p}{\phi}
\newcommand{\e}{\eta}
\newcommand{\la}{\lambda}
\newcommand{\D}{\Delta}
\newcommand{\ep}{\epsilon}
\newcommand{\G}{\Gamma}
\newcommand{\g}{\gamma}
\newtheorem*{theorem*}{Theorem}
\newtheorem*{Corollary*}{Corollary}
\numberwithin{equation}{section}
\title[The $L^p$-spectrum of the Laplacian on forms]{ The $L^p$-spectrum of the Laplacian on forms over warped products and Kleinian groups}
\author{PETROS SIASOS} 
\address{PETROS SIASOS, DEPARTMENT OF MATHEMATICS AND STATISTICS, UNIVERSITY OF CYPRUS, P.O. Box 20537, CY-1678 NICOSIA, CYPRUS}
\email{siasos.petros@ucy.ac.cy}
\begin{document}

\keywords{$L^p$-spectrum, Laplacian on forms, warped products, Kleinian groups, This work was partially supported by the University of Cyprus Internal Grant of Nelia Charalambous.}
\date{\today}

\maketitle

\begin{abstract}
In this article, we generalize the set of manifolds over which the $L^p$-spectrum of the Laplacian on $k$-forms depends on $p$. We will consider the case of manifolds that are warped products at infinity and certain quotients of Hyperbolic space. In the case of warped products at infinity we prove that the $L^p$-spectrum of the Laplacian on $k$-forms contains a parabolic region which depends on $k$, $p$ and the limiting curvature $a_0$ at infinity. For   $M=\mathbb{H}^{N+1}/\Gamma $ with $\G$ a geometrically finite group such that $M$ has infinite volume and no cusps, we prove that the $L^p$-spectrum of the Laplacian on $k$-forms is a exactly a parabolic region together with a set of isolated eigenvalues on the real line.
\end{abstract}

\section{Introduction}

The study of the $L^p$-spectrum of the Laplace-Beltrami-operator on Riemannian manifolds is an active research area in the last decades. As the following historical references show, the $L^p$-spectrum of the Laplace-Beltrami operator on Riemannian manifolds may depend on $p$ and this dependence can reflect the geometric structure and properties of the manifold. It is for example connected to the volume growth of the manifold.

In the case of operators on functions we have two guiding examples. We know that the spectrum of the Laplace-Beltrami operator on Euclidean space is $p$-independent, whereas the spectrum on the Hyperbolic space depends on $p$. Hempel and Voigt \cite{hempel1986spectrum}
 studied the $L^p$-spectrum on Schrodinger operators on Euclidean spaces and found sufficient conditions on the potential so that the spectrum is independent of $p$. Sturm \cite{sturm1993lp} studied the $L^p$-spectrum for a class of uniformly elliptic operators on functions over open manifolds and showed the $L^p$-independence of the spectrum on such spaces, whenever the volume of the manifold has uniformly sub-exponentially volume growth and Ricci curvature bounded below. In the case of negatively curved manifolds,  Davies, Simon and Taylor \cite{davies1988lp} examine the $L^p$-spectrum on quotients of the Hyperbolic space $\mathbb{H}^{N+1}/ \G$ and showed the $L^p$-dependence of the spectrum. Specifically, under the assumptions that $\mathbb{H}^{N+1}/ \G$ is of finite volume or has no cusps, they completely determined the $\mathbb{H}^{N+1}/ \G$-spectrum being a parabolic region of the Complex plane for $p \neq 2$ together with a finite set of isolated eigenvalues and this reduced to a closed subset of  $\mathbb R$, for $p=2$ which is an interval with a finite set of eigenvalues. Other examples of negatively curved manifolds where the $L^p$-spectrum of the Laplacian on functions depends on $p$ can be found in the work of Taylor \cite{Taylor1989LpestimatesOF} and Weber \cite{weber2007heat}. They
studied the $L^p$-spectrum of the Laplace-Beltrami operator for quotients of symmetric spaces. These spaces are called locally symmetric spaces. More recently, Charalambous and Rowlett \cite{charalambous2023laplace} studied the $L^p$-spectrum on conformally compact manifolds.

The $L^p$-spectrum has also been studied for the Laplacian on forms on Riemannian manifolds. Donnelly in \cite{donnelly1981differential} computed the $L^2$-spectrum of the Laplacian on forms of Hyperbolic space. Mazzeo in \cite{mazzeo1988hodge} computed the $L^2$ essential spectrum of the Laplacian on forms of a conformally compact metric. Antoci also computed in \cite{antoci2004spectrum} the $L^2$essential spectrum of the Laplacian on forms, for a class of warped product metrics. At the same time, Charalambous \cite{charalambous2005lp} proved the $L^p$-independence of the spectrum of the Laplacian on forms on non-compact manifolds, for $1 \leq p \leq \infty$. More specifically, she showed that under the assumptions that the Ricci curvature is bounded below, the volume growth is uniformly subexponential and the Weitzenbock tensor is bounded below the $L^p$-spectrum of the Laplacian on forms is independent of $p$ for $1 \leq p \leq \infty$.  More recently, Charalambous and Lu \cite{charalambous2024lpspectral} computed the $L^p$-spectrum of the Laplacian on forms of Hyperbolic space, for $1 \leq p \leq \infty$, and proved that the form spectrum is $p$-dependent on this negatively curved space.

In this article our main goal is to generalize the set of manifolds over which the spectrum of the Laplacian on $L^p$ integrable $k$-forms, $\sigma(p,k,\D)$ depends on $p$. We will consider the case of manifolds that are warped products at infinity and certain quotients of Hyperbolic space. The Laplacian on $k$-forms has a strong connection to the geometry and topology of a manifold, and hence is is usually more difficult to obtain results for the spectrum on forms in comparison to the spectrum on functions. This is the main reason why we initially concentrate on manifolds with a more rigid structure.

This article is organized as follows. In Section 2 we consider warped products at infinity. These are manifolds such that outside a compact set $K$, $M\setminus K$ is of the form $(c_0, \infty) \times N$ with metric $g=dr^2+f^2(r)g_{N},$
    where $f \in C^\infty(c_0,\infty)$ is the warping function and $N$ is an $(n-1)$-dimensional compact manifold.  We consider the class of warped product metrics with warping function $f$ in the following class:
\begin{equation*}
    \begin{split}
        B=\{f \in C^2(a,\infty)& : \frac{f''}{f}=a_0 +o(1), \\
        & \left(\frac{f'}{f}\right)^2=a_0+o(1), \text{as}  \; r \rightarrow \infty,\text{with} \;a_0>0 \\
       &  \; \; \text{and}\; f \to \infty, \text{as}  \; r \rightarrow \infty  \}
        \end{split}
\end{equation*}
and we prove that over such $M$ the $L^p$ spectrum of the Laplacian on $k$-forms contains the parabolic region 
$$
Q_{p,k}=\{a_0\left(\frac{n-1}{2}-k\right)^2+z^2: |\text{Im} (z)|\leq\sqrt{a_0}(n-1)\left|\frac{1}{p}-\frac{1}{2}\right| \},
$$
which depends on $k,p$ and $a_0$.
\begin{theorem}\label{thm1}
   Let $M$ be a warped product at infinity where the warping function $f \in B$. For any $0\leq k\leq \frac{n}{2} $ and $1\leq p \leq 2$, the $L^p$ spectrum of the Hodge Laplacian, $\sigma (p,k, \Delta)$ contains $Q_{p,k}$. The remaining cases for $p$ and $k$  are given by duality as $\sigma (p,k, \Delta)=\sigma (p,n-k, \Delta)$ for $n/2\leq k\leq n$ and $\sigma(p,k, \Delta)=\sigma (p^*,k, \Delta)$ whenever $ \frac{1}{p}+\frac{1}{p^*}=1$.
\end{theorem}
In Section 3 we prove that the $L^p$ spectrum of the Laplacian on $k$-forms is contained in a parabolic region that also depends on the bottom of the $L^2$-spectrum (see Proposition \ref{Spectrum}). In the particular case that $f(r)\sim ce^{\sqrt{a_0}r}$ and the Laplacian on $k$-forms has no isolated eigenvalues of finite multiplicity, we prove that the $L^p$-spectrum is the parabolic region $Q_{p,k}$ (see Theorem \ref{ResultA}).

In Section 4 we study the $L^p$ spectrum of the Laplacian on forms on quotients of Hyperbolic space $M=\mathbb{H}^{N+1}/\Gamma $. In \cite{davies1988lp} Davies, Simon and Taylor studied the $L^p$ spectrum of the Laplace-Beltrami operator $\D_\G$ over functions on non-compact quotients $M=\mathbb{H}^{N+1}/\Gamma $, where $\G$ is a geometrically finite group. Specifically, if $M$ is either of finite volume or cusp-free, they determine explicitly the $L^p$-spectrum of $\D_\G$ for $1 \leq p \leq \infty$, proving that it is a parabolic region together with a finite set of isolated eigenvalues. In this section we generalize their theorem to the Laplacian on forms $\vec{\D}_{\G}$ in the case when $M$ has no cusps. In order to do this we generalize many of the results from \cite{davies1988lp} concerning properties of the heat semigroup, the heat kernel and the resolvent operator for the Laplacian on functions to those corresponding to the Laplacian on forms.

 We prove that over   $M=\mathbb{H}^{N+1}/\Gamma $ with $\G$ a geometric finite group such that $M$ has infinite volume and no cusps, the $L^p$-spectrum of the Laplacian on $k$-forms is a exactly a parabolic region
  \begin{equation} \label{Q'_pk}
     Q'_{p,k}=\{\left(\frac{N}{2}-k\right)^2+z^2: |\text{Im} z|\leq N|\frac{1}{p}-\frac{1}{2}| \},
 \end{equation}
 together with a set of isolated eigenvalues on the real line.

\begin{theorem} \label{Theorem5.2}
   Let $M=\mathbb{H}^{N+1} / \G$ , where $\G$ is a geometrically finite group and $M$ has infinite volume and no cusps. In addition, assume that the set of isolated eigenvalues in the spectrum of Laplacian on $L^2$-integrable $k$-forms is finite, and consists of the points $\{E_0,\dots,E_m\}$ for any $k\neq \frac{N+1}{2}$. Then, for $1\leq p <\infty $ and $k< \frac{N+1}{2}$
   $$\sigma(k,p,\D_\G)=\{E_0,\dots,E_m\} \cup Q'_{p,k}
   $$
   and for $k> \frac{N+1}{2}$ 
   $$
   \sigma(k,p,\D_\G)= \sigma(n-k,p,\D_\G).
   $$
\end{theorem}
Finally the last Section is an Appendix on which we have included some necessary propositions that will be used throughout this article.
\section*{Acknowledgements}
The results of this paper were obtaining during my PhD studies at the University of Cyprus. I would like to thank my advisor Nelia Charalambous for suggesting me to work on this problem and also for her guidance and support for the completion of this project.
\section{The $L^p$ spectrum for a class of warped product metrics part I}

Before we start this section let us recall the definition of the $L^p$ spectrum. On a complete non compact Riemannian manifold $M$,  the  Laplacian over smooth differential $k$-forms,  is defined by $\D=d\delta+\delta d$, where $d$ is the exterior differentiation and $\delta$ is the co-differential operator.  On $L^2(M)$ the Laplacian on $k$-forms $\D$ is defined via Friedrichs extension (see Theorem 1.2.8 \cite{davies1989heat}). On $L^p(M)$ the Laplacian on $k$-forms $\D$ is defined as the generator of the heat semigroup on $L^p(M)$ (see Theorem 1.4.1 \cite{davies1989heat}). We denote the spectrum of the Laplacian on $L^p$ integrable $k$-forms by $\sigma(k,p,\D)$. Let us note that  $\sigma(k,p,\D)=\sigma(k,p',\D)$ for $1/p+1/p'=1$.

In this section, the space $M$ will be a generalization of warped product manifolds. More specifically we will study manifolds which are warped products at infinity. Let us be more precise.

\begin{Definition}
    We say that M is a warped product at infinity, if outside a compact set $K$, $M\setminus K$ is of the form $(c_0, \infty) \times N$ with metric 
    $$
g=dr^2+f^2(r)g_{N},
$$
    where $f \in C^\infty(c_0,\infty)$ is the warping function and $N$ is an $(n-1)$-dimensional compact manifold. 
\end{Definition}
The type of warping function $f$ determines the geometry of the manifold at infinity, and hence its spectrum. For example if $c_0>0$, $f=r$ with $N=S^{n-1}$, then $M$ is Euclidean at infinity and if $f=sinhr$ with $N=S^{n-1}$ as well it is hyperbolic. We will provide a large class of functions $f$ for which the $L^p$ spectrum contains a parabolic region. 
We will consider the set of functions:
\begin{equation*}
    \begin{split}
        B=\{f \in C^2(a,\infty)& : \frac{f''}{f}=a_0 +o(1), \\
        & \left(\frac{f'}{f}\right)^2=a_0+o(1), \text{as}  \; r \rightarrow \infty,\text{with} \;a_0>0 \\
       &  \; \; \text{and}\; f \to \infty, \text{as}  \; r \rightarrow \infty  \}.
        \end{split}
\end{equation*}
A prototype example of such $f$ are $f=e^{\sqrt{a_0}r}$ and $f=sinh (\sqrt{a_0}r)$, which give us a manifold which is hyperbolic with Ricci curvature in the radial direction $-(n-1)a_0$. In other words, we consider manifolds which are asymptotically hyperbolic and with infinity volume since $f \to \infty$.

Let $P_{p,k}$ be the curve in the complex plane,
$$
P_{p,k}=\left\{ -a_0\left[\frac{n-1}{p}-k+is\right] \left[(n-1)\left(\frac{1}{p}-1\right)+k+is\right], s \in \mathbb{R}\right\}
$$
with $a_0>0$. Denote the parabolic region to the right of the curve $P_{p,k}$ by $Q_{p,k}$.  A straightforward  computation gives.
\begin{lemma}\label{lemma2}
    The parabolic region $Q_{p,k}$ can be expressed as,
 $$
Q_{p,k}=\{a_0\left(\frac{n-1}{2}-k\right)^2+z^2: |\text{Im} (z)|\leq\sqrt{a_0}(n-1)\left|\frac{1}{p}-\frac{1}{2}\right| \}.
$$
\end{lemma}

In this section we construct approximate eigenforms in order to prove that the $L^p$ spectrum of the Hodge Laplacian on $k$-forms contains the parabolic region $Q_{p,k}$ whenever $M$ is a warped product at infinity with warping function $f$. To show that $\la \in \sigma(p,k,\D)$, it suffices to find a sequence of approximate eigenforms i.e. for every $\ep>0$ find a $k$-form $\omega_\ep$ such that 
$$
\left\lVert \Delta \omega_\ep - \lambda \omega_\ep\right\rVert_{L^p} \leq \epsilon \left\lVert \omega_\ep\right\rVert_{L^p},
$$
(see also \cite{charalambous2024lpspectral} of an if and only if criterion).

In order to construct this approximate eigenforms we need to know the action of the Laplacian operator on certain canonical type of forms. We will use the following computation from \cite{antoci2004spectrum}. 

\begin{prop} \label{Delta expression} [(3.7) and (3.8) in \cite{antoci2004spectrum}]
Let $\omega_1=h_1(r)\e_1$, $\omega_2=h_2(r)\e_2$,  where $\e_1$ is a $k$-form on $N$ and $\e_2$ is a $(k-1)$-form on $N$, and $h_1,h_2$ are smooth functions of $r$. Then, the $k$-form $\omega=\omega_1+\omega_2 \wedge dr$ on $M$, satisfies
\begin{equation*}
    \begin{split}
        \D \omega & =h_1 f^{-2} \D_N \e_1 +h_2f^{-2} (\D_N \e_2) \wedge dr +(-1)^k2 h_1 f' f^{-3}( \delta_N \e_1 )\wedge dr \\ 
        &+(-1)^k 2h_2 f' f^{-1} d_N \e_2 -[h_1''+(n-2k-1)h_1'f'f^{-1}] \e_1\\
        & - [h_2'' +(n-2k+1)(h_2f' f^{-1})']\e_2 \wedge dr.\\
    \end{split}
\end{equation*}
\end{prop}
 From the above decomposition of $\D$ we easily get the following.

\begin{Corollary} \label{Corollary 3.1}
 Let $\e$ be a closed $(k-1)$-eigenform over $N$ with $\D_N \e =\lambda_0 \e$. Then, the $k$-form $\omega=h_2(r) \e \wedge dr$ over $M$ satisfies
        $$
         \D(h_2(r) \;\e \wedge dr) = \D_2(h_2(r)) \; \e \wedge dr,
        $$
        where 
     \begin{equation}
\Delta_2(h_2(r))=-\left[h_2''(r)+(n-2k+1)\left(h_2(r)\frac{f'(r)}{f(r)}\right)'\right]+\lambda_0 \frac{h_2(r)}{f^2(r)}. \label{eq1forms}
\end{equation}
   
\end{Corollary}
For the warping function $f$ and a smooth function $\p=\p(r)$ we make the following computations.
\begin{equation}
    \Delta_2 f^\mu =f^\mu \left[-(\mu +n-2k+1)\frac{f''}{f}-(\mu-1)(\mu+n-2k+1)\left(\frac{f'}{f}\right)^2 +\frac{ \lambda_0}{f^2}\right] \label{eq2}
\end{equation}
and
\begin{equation}
\begin{split}
    \Delta_2 (\phi f^\mu) =& \p \D_2(f^\m)-\left[\p''f^\m +2 \p' (f^\m)' +(n-2k+1) \p ' f^{\m-1} f'\right]\\
    =&-f^\mu \left[ (\mu +n-2k+1) \phi \frac{f''}{f}+ (\mu -1)(\mu +n-2k+1) \phi \left(\frac{f'}{f}\right)^2 \right.\\ 
    +&\left. \phi''+(2\mu+n-2k+1)\phi' \frac{f'}{f} -\lambda_0 \frac{\phi}{f^2}\right].
\end{split}
     \label{eq3}
\end{equation}

Now we prove Theorem \ref{thm1}.

\textit{Proof of Theorem \ref{thm1}}.
Let us note that by our assumption $f \to \infty$. Let $\m$ be a complex number. Note that in the particular case $$ \frac{f''}{f}=a_0, \; \; \left(\frac{f'}{f}\right)^2 =a_0,$$
(\ref{eq2}) gives
$$\Delta_2 f^\mu =-f^\mu[a_0 \mu (  \mu +n -2k +1)]+ \frac{\la_0 f^\m}{f^2}.$$
If $f \rightarrow \infty$, the last term is of lower order. Hence, the candidate points for the spectrum are
    \begin{equation}
        \label{eqlabda} \lambda=-a_0 \mu (  \mu +n -2k +1)=-a_0( \mu +n -2k +1) -a_0(\mu -1)(\mu +n-2k +1).
    \end{equation}
We will show that any $\la$ such as the one above belongs to $\sigma (p,k, \Delta) $. To achieve this we consider approximate eigenforms of the type $\omega=\p f^\m  \e\wedge dr$ where $\p=\p (r)$ is smooth and has compact support in $(c_0,\infty)$, $\m \in \mathbb{C}$ and $\e$ is a smooth closed $(k-1)$ eigenform on $N$ with eigenvalue $\la_0$. Set $\la$ as in (\ref{eqlabda}). By Corollary  \ref{Corollary 3.1} $\D (\omega)=\D_2(\p f^\m) \;\e \wedge dr$.  Using (\ref{eq3}) and the triangle inequality we get
   \begin{equation}\label{eq fiveterms}
   \begin{split}
& \lVert \Delta_p \omega - \lambda \omega \rVert_p^p
 =\lVert \Delta_2 (\phi f^\mu) \e\wedge dr - \lambda \phi f^\mu \e\wedge dr \rVert^p_p\\
&  = \left\Vert   f^\mu \e\wedge dr \left[-(\mu-1)(\mu+n-2k+1)\phi \left(\frac{f'}{f}\right)^2 -(\mu +n-2k+1)\phi \frac{f''}{f} \right] \right.\\
&  -f^\mu \e\wedge dr \left[\phi''+(2\mu+n-2k+1)\phi'\frac{f'}{f}-\lambda_0\frac{\phi}{f^2}\right]\\
& \left.  +a_0(\mu +n-2k+1)\phi f^\mu \eta \wedge dr+a_0(\mu-1)(\mu+n-2k+1)\phi f^\mu \e\wedge dr \right\Vert^p_p \\
& = \left\Vert -(\m-1)(\m +n-2k+1)\p \left[\left(\frac{f'}{f}\right)^2 - a_0\right]f^\m \e\wedge dr \right.\\
&  -(\m+n-2k+1)\p\left(\frac{f''}{f}-a_0\right)f^\m \e\wedge dr\\
& \left. -f^\m \p'' \e\wedge dr -f^\m \e\wedge dr (2\m +n-2k+1) \p'\frac{f'}{f}+f^\m \e\wedge dr \lambda_0 \frac{\p}{f^2} \right\Vert^p_p\\
&  \leq |(\m-1)(\m+n-2k+1)| \; \left\Vert \p \left[\left(\frac{f'}{f}\right)^2 - a_0\right]f^\m \e\wedge dr \right\Vert^p_p \\
&  + |\m +n-2k+1|\left\lVert \p \left(\frac{f''}{f}-a_0\right)f^\m \e\wedge dr \right\rVert^p_p\\
&  + \lVert \p''f^\m \e\wedge dr \rVert^p_p + |2\m+n-2k+1| \rVert \p' \frac{f'}{f}f^\m \e\wedge dr \rVert^p_p+ \rVert \la_0 \frac{\p}{f^2}f^\m \e\wedge dr \rVert^p_p\\
&  =I +II+III+IV+V.
\end{split}
\end{equation}

We set $\m= -\frac{n-1}{p}+(k-1)+is$, for $s\in\mathbb{R}$. Note that $$| \e\wedge dr |_M=|\e|_{N}f^{-(k-1)},$$
hence
\begin{equation}\label{eq_omeganorm}
    |\omega|_M=f^{\text{Re}(\m)-k-1}|\p||\e|_{N}=f^{-\frac{n-1}{p}}|\p|| \e|_{N},
\end{equation}
for our $\m$. Fix $\ep>0$. Let $A_\ep$ ,$B_\ep$ such that  $B_\ep>A_\ep >C+1$.We will take cut-off functions $\p_{\ep} : \mathbb{R}\rightarrow \mathbb{R}$ such that $\p_\ep \in C^\infty (\mathbb{R}), \;  \text{spt}\p_\ep \subset [A_\ep-1,B_\ep] \subset(c_0,\infty), \; \p_\ep =1\; \text{on} \;[A_\ep, B_\ep]$ and  $|\p_\ep'|, |\p_\ep''| \leq C$ for every $\ep>0$.
Observe that, 
$$
\int_0^\infty|\p_\ep'|^pdr\leq \int_{A_\ep-1}^{A_\ep} |\p_\ep'|^pdr+ \int_{B_\ep}^{B_\ep+1} |\p_\ep'|^pdr \leq C
$$
and
$$
\int_0^\infty|\p_\ep''|^pdr \leq \int_{A_\ep-1
}^{A_\ep} |\p_\ep''|^pdr+ \int_{B_\ep}^{B_\ep+1} |\p_\ep''|^pdr \leq C,
$$
where $C$ are uniform constants independent of $\ep$. From now on $C$ will denote a uniform constant which does not depend on $\ep$, and might differ from one line to the next. Note that the volume element on $M\setminus K$ is $dV_g=f^{n-1} dr d\sigma$ where $d\sigma$ is the volume element on $N$.  Using (\ref{eq_omeganorm}) 
we compute
\begin{equation}\label{eq4}
\begin{split}
    \lVert\omega_\ep \rVert^p_p &=\int_M |\phi_\ep|^p f^{p(\text{Re}\m)}|\e \wedge dr|^p dV_g =
  \int_{A_\ep-1}^{B_\ep+1} \int_{N}f^{p(\text{Re}\m)} |\p_\ep|^p |\e \wedge dr |^pf^{n-1} drd\sigma \\
    &=\int_{A_\ep-1}^{B_\ep+1}C|\p_\ep|^pdr >\int_{A_\ep}^{B_\ep}Cdr=(B_\ep-A_\ep)C.
\end{split}
\end{equation}
We will show that each one of the five terms in (\ref{eq fiveterms}) is uniformly bounded for our $\omega$. First we get 
\begin{equation}\label{eq5}
\begin{split}
IV \leq C \int_M |\p_\ep'|^p\left(\frac{f'}{f}\right)^p|f^\m |^p\; |\e \wedge dr|^p dV_g&= \int_{\text{spt}\p_{\ep}'}C|\p_\ep'|^p\left(\frac{f'}{f}\right)^pdr \\
 & \leq  C\int_{[A_{\ep}-1,A_\ep]\bigcup[B_\ep,B_\ep+1]} |\p_\ep'|^pdr\leq C,   
\end{split}
\end{equation}
where we have used that $\left(\frac{f'}{f}\right)^p$ is bounded at infinity and hence it is uniformly bounded on $(c_0,\infty)$. Similarly, $|\p_\ep''|$ is bounded and supported on a set of length 2, hence $III\leq C$.

Now we estimate $(I)$.
\begin{equation}
\begin{split}
        I & =C \left\Vert \p_\ep \left[\left(\frac{f'}{f}\right)^2-a_0\right]f^\m \e \wedge dr \right\Vert^p_p = C \int_M \p_{\ep}^p \left|\left(\frac{f'}{f}\right)^2-a_0 \right|^p f^{p (\text{Re}\m)} |\e \wedge dr|^pdV_g \\
        & \leq C \sup_{[A_\ep-1,B_\ep+1]} \left| \left(\frac{f'}{f}\right)^2-a_0 \right|^p\int_M \p_{\ep}^p f^{p (\text{Re}\m)} |\e \wedge dr|^pdV_g \\
        & =C \sup_{[A_\ep-1,B_\ep+1]} \left| \left(\frac{f'}{f}\right)^2-a_0 \right|^p \|\omega_\ep\|^p_p.
\end{split}
\end{equation}

Similarly, 
$$
II\leq C \sup_{[A_\ep-1,B_\ep+1]}\left|\frac{f''}{f}-a_0 \right|^p\lVert \omega_\ep\rVert^p_p,
$$

and
$$
V \leq C \sup_{[A_\ep-1,B_\ep+1]}\left(\frac{1}{f^{2p}}\right)\lVert \omega_\ep\rVert^p_p.
$$

By our assumptions $\frac{f''}{f} \rightarrow a_0,$ $ \left( \frac{f'}{f}\right)^{1/2} \rightarrow a_0$ and $ f \rightarrow \infty$, hence we can find $A_\ep$ large enough such that 
$$
\sup_{[A_\ep-1,B_\ep+1]}\left|\left(\frac{f'}{f}\right)^2-a_0\right|^p, \; \sup_{[A_\ep-1,B_\ep+1]} \left|\frac{f''}{f}-a_0\right|^p, \; \sup_{[A_\ep-1,B_\ep+1]}\left(\frac{1}{f^{2p}}\right) \leq \ep.
$$
As a result 
$$
I+II+V \leq C \;\ep \|\omega_\ep\|^p_p,
$$
with $C$ independent of $\ep$. In addition, for any $A_\ep$ we can chose $B_\ep$ large enough such that 
$$
B_\ep-A_\ep> \frac{1}{\ep}.
$$
Hence, by (\ref{eq4}) and (\ref{eq5})
$$
III +IV \leq C = C\;\ep\; \frac{1}{\ep} \leq C \; \ep \|\omega_\ep \|^p_p.
$$
Since, $\ep$ was arbitrary we conclude that 
$$
I +II+III+IV+V\leq \ep\lVert \omega_\ep\rVert^p_p.
$$
So, we have shown that the points
$$
\la=-a_0\left(-\frac{n-1}{p}+ (k-1)+is\right)\left(-\frac{n-1}{p}+(n-k)+is\right), \;s \in \mathbb{R}
$$
belong to $\sigma (p,k, \Delta)$. Setting $k=n-m$, for $0\leq m\leq n/2$, in the above equation and changing sign in both brackets we get

$$
\la=-a_0\left[(n-1)\left(\frac{1}{p}-1\right)+m-is\right]\left[\frac{n-1}{p}-m-is\right], \; s \in \mathbb{R},
$$
which are exactly the points of $P_{p,m}$.
Thus, for $0\leq m\leq n/2$ we have shown that $P_{p,m} \subset \sigma(p,n-m,\D)=\sigma(p,m,\D)$, where the equality follows by Poincare duality. 

Observe that $\bigcup_{p\leq q \leq 2}P_{q,m}=Q_{p,m}$. Moreover, by $\sigma(q,m,\D) \subset \sigma(p,m,\D)$ for all $1 \leq p \leq q \leq 2$ we get $Q_{p,m} \subset \sigma(p,m,\D)$ for any $1 \leq p\leq 2$ and $0\leq m\leq \frac{n}{2}$.

\section{The $L^p$ spectrum for a class of warped product metrics part II}
In this section, we show that the the $L^p$ spectrum is contained within a parabolic region. This is a result that depends on the rate of the volume growth of the manifold at infinity, which is defined as follows.
\begin{Definition}
    The exponential rate of volume growth of $M$, denoted by $\gamma$, is the infimum of all real numbers satisfying the property: for any $\ep>0$, there is a constant $C$, depending only on $\ep$ and the dimension of $M$, such that for any $p \in M$ and any $R \geq 1$, we have
    $$
    V_p(R)\leq C V_p(1)e^{(\gamma+\ep)R},
    $$
where $V_p(R)$ denotes the volume of the ball of radius $R$ at $p$.
\end{Definition}

\begin{theorem} \label{thm2} \cite{charalambous2024lpspectral}
Let $M$ be a complete manifold over which the Ricci curvature and the Weitzenbock tensor on $k$-forms are bounded below. Denote by $\gamma$ the exponential rate of volume growth of $M$ and $\la_1$ the infimum of the spectrum of the Laplacian on $L^2$ integrable $k$-forms. Let $1\leq p \leq \infty$, and $z$ be a complex number such that 
$$
|\text{Im}(z)| >\gamma | \frac{1}{p}-\frac{1}{2}|.
$$
Then 
$$
(\D-(\la_1+z^2))^{-1}
$$
is a bounded operator on $L^p(\Lambda^k(M))$.
\end{theorem}
In other words, the $L^p$ spectrum of $\D$ on $k$-forms contained in the parabolic region $ \{\la_1+z^2:|\text{Im}(z)| >\gamma | \frac{1}{p}-\frac{1}{2}| \}$.  Now let us return to the set of the manifolds $M$ that we study in this section, that is, warped products at infinity with warping function $f \in B$. In order to apply the previous Theorem in our case, we have to show first that for $M$ the Ricci curvature and the Weitzenbock tensor on $k$-forms are bounded below and in the sequel to compute $\g$.  In the following Proposition we prove that the Ricci curvature of $M$ is bounded below. We compute also an upper bound for $\gamma$.
\begin{prop} \label{prop2}
    Let $M$ be a warped product at infinity where the warping function $f \in B$. Then the Ricci curvature of $M$ is bounded below, and the exponential of volume growth of $M$ is at most $(n-1)\sqrt{a_0}$.
\end{prop}
\begin{proof}
    The technique that we use is based  on the proof of Proposition 1 in \cite{sturm1993lp}, but here we generalize it to manifolds whose Ricci curvature is bounded below by a negative constant, whereas Sturm only considers asymptotically nonnegative Ricci curvature. By Proposition \ref{prop compact} we have that all sectional curvatures of $M$ lie between 
$$
-\frac{f''}{f}, \;\;\frac{\bar C-(f')^2}{f^2},
$$
where $\bar C$ reflects the sectional curvatures of $N$. As a result, on $M\setminus K$ the sectional curvatures tend to $-a_0$ as $r\to \infty$. This implies that, for any $\ep>0$, there is a compact set $K_\ep \subset M$ with smooth boundary, such that the sectional curvatures on $M\setminus K_\ep$ are bounded below by $-(a_0+\ep)$. Therefore, the Ricci curvature is bounded below by $-(n-1)(a_0+\ep)$ on $M\setminus K_\ep$. Let $R=R_\ep$ be the diameter of $K_\ep$.

For a point $x \in M$, let $s=s_\ep(x)=d(x,K_\ep)$ and $t=t_\ep(x)=s+R_\ep$. By definition, $t-s$ does not depend on $x \in M$. Now, by the compactness of $K_\ep$, since the Ricci curvature is uniformly bounded from below on $M\setminus K_\ep$, we get that it is uniformly bounded from below on all of $M$. Thus, there exists a $K>0$ such that the Ricci Curvature is bounded from below by $-(n-1)K^2$ on all of $M$. Clearly, $K \geq \ \sqrt{a_0 +\ep}$.

We introduce the Sturm-Liouville equation,
\begin{equation}
    \label{sturm equation}
    u''(r)+q(r)u(r)=0,
\end{equation}

with initial conditions $u(0)=0$, $u'(0)=1$, where $q$ is defined by
\begin{equation} \label{eqsturmmmm}
  q(r) = 
     \begin{cases}
      -(a_0+\ep) &\quad \text{for} \; r \in [0,s) \\
        -K^2 &\quad \text{for} \;r \in [s,t) \\
         -(a_0+\ep) &\quad \text{for} \; r \in [t,\infty) \\
     \end{cases}.  
\end{equation}

Denote by $M^n_q$ the complete simply connected space with piecewise constant curvature as defined by $q$. Note that the volume element  of $ M^n_q$ is $u(r)^{n-1}dr\wedge d \theta$, where $u(r)$ is the solution of (\ref{sturm equation}). 

From Bishop Volume Theorem (see \cite{li2012geometric} p.14) we have

$$
\frac{V_x(r)}{V_x(1)}\leq \frac{\bar{V}(r)}{\bar{V}(1)}
$$
where $V, \bar{V}$ are the volumes of the geodesic ball of $M$ and $ M_{q}^{n}$ respectively. We compute 

$$
\bar{A}(r')=\int_{S^{n-1}} J_q(r',\theta) d\theta = \int_{S^{n-1}} u(r')^{n-1} d \theta=  u(r')^{n-1} \int_{S^{n-1}} d\theta
$$
and
$$
\bar{V}(r)= \int_{S^{n-1}} d\theta \int^r_0u(r')^{n-1}dr'.
$$
Therefore,
\begin{equation}\label{volineq}
    V_x(r)\leq V_x(1)\frac{\int^r_0u(r')^{n-1}dr'}{\int^1_0u(r')^{n-1}dr'}.
\end{equation}
The estimate of the above fraction will give us the upper volume estimate of $V_x(r)$.

By solving  (\ref{sturm equation}) with $q$ as in (\ref{eqsturmmmm}) and the properties of the Sturm-Liouville's equation we get that 
$$
u(r)\geq \frac{1}{\sqrt{a_0+\ep}}sinh(r\sqrt{a_0+\ep}), 
$$
on $(0,\infty)$, uniformly on  $M$. 
Thus, 

$$
\int^1_0u(r')^{n-1}dr'\geq \int_0^1 \left(\frac{1}{\sqrt{a_0+\ep}}sinh(r'\sqrt{a_0+\ep})\right)^{n-1} dr'=C_\ep.
$$
Combining with (\ref{volineq}) we find

$$
V_x(r) \leq V_x(1) C'_{\ep} \int_0^r u(r')^{n-1}dr'.
$$
Thus it remains to estimate 

$$
\int_0^r u(r')^{n-1}dr'.
$$
By solving the Sturm-Liouville's equation iteratively on each interval, we have

\begin{equation} \label{snk}
  u(r) \leq  
     \begin{cases}
      \frac{1}{\sqrt{a_0+ \ep}} e^{r \sqrt{a_0+ \ep} }  &\quad \text{if} \; r \in [0,s) \\
       \frac{1}{\sqrt{a_0+ \ep}} e^{s \sqrt{a_0+ \ep} } e^{K(r-s)}  &\quad \text{if} \; r \in [s,t) \\
         \frac{K}{a_0+ \ep} e^{s \sqrt{a_0+ \ep} } e^{K(t-s)} e^{(r-t) \sqrt{a_0+ \ep} } &\quad \text{if} \; r \in [t,\infty) \\
     \end{cases}.
\end{equation} 
Note that 
\begin{equation*}
    \begin{split}
        e^{ s\sqrt{a_0+ \ep} }e^{K(t-s) +(r-t)\sqrt{a_0+ \ep} }&=e^{(K-\sqrt{a_0+ \ep})(t-s)} e^{r\sqrt{a_0+ \ep} }\\
        &=e^{(K-\sqrt{a_0+ \ep})R_\ep} e^{r\sqrt{a_0+ \ep} }
    \end{split}
\end{equation*}
for $r \in [t,\infty)$. Moreover,

$$
e^{s\sqrt{a_0+ \ep} } e^{K(r-s)} \leq e^{r\sqrt{a_0+ \ep} } e^{KR_\ep},
$$
for $r \in [s,t)$. As a result 
$$
u(r) \leq C_\ep e^{r\sqrt{a_0+ \ep}},
$$
for all $r \in [0,\infty)$, since $R_\ep$ is independent of $x$.
We can now estimate
$$
\int_0^r u(r')^{n-1}dr' \leq C_\ep \int_0^r \left(e^{r' \sqrt{a_0+ \ep}}\right)^{n-1} dr'=C_\ep e^{(n-1)r \sqrt{a_0+ \ep}}.
$$
Combining the above, we get that the fraction in (\ref{volineq}) is uniformly bounded above in $M$ by $C_\ep e^{(n-1)r\sqrt{a_0+\ep}}$. Thus, since $\ep>0$ was arbitrary, we get that the exponential rate of volume growth, $\gamma$ satisfies $\g \leq (n-1)\sqrt{a_0}$.
\end{proof}
The above Proposition combined with our computation of the $L^p$ spectrum will give us the exact value of the exponential rate of volume growth $\g$ of $M$. More precisely we have.
\begin{prop} \label{propprop}
   Let $M$ be a warped product at infinity where the warping function $f \in B$. Then $\gamma=(n-1) \sqrt{a_0}$.
\end{prop}
\begin{proof}
    We proceed as in \cite{charalambous2023laplace}. By Proposition \ref{prop2}, the Ricci curvature of $M$ is bounded below, and by Proposition \ref{Proposition2.6} the curvature operator on $M$ is bounded below. Hence by Corollary 2.6 in \cite{MR0454884} the Weitzenbock tensor on $k$-forms is also bounded below. By Theorem \ref{thm2} we find that the resolvent set of the Hodge Laplacian acting on $L^p$ $k$-forms contains the set
\begin{equation} \label{eq55}
    A=\{\la_1+z^2:|\text{Im}(z)|>|\frac{1}{p}-\frac{1}{2}|\gamma\},
\end{equation}
where $\la_1$ is the bottom of the $L^2$ spectrum of the Laplacian on $k$-forms. Thus, we have 
\begin{equation}\label{eq6} 
  \sigma (p,k,\D) \subset \mathbb{C} \setminus A =A^{\complement}= \{\la_1+z^2:|\text{Im}(z)|\leq |\frac{1}{p}-\frac{1}{2}|\gamma\}.  
\end{equation}
From Theorem \ref{thm1} we have,
\begin{equation} \label{Q_pk and spectrum}
    Q_{p,k} \subset \sigma (p,k,\D).
\end{equation}
We proceed by contradiction. Assuming $\gamma <(n-1)\sqrt{a_0}$ we will show that this forces certain points of $Q_{p,k}$ and hence $\sigma (p,k,\D)$ to lie outside  $\mathbb{C} \setminus A$ giving the desired contradiction.
For simplicity we take $p=1$. In this case, we have,
$$
A^{\complement} = \{\la_1+z^2:|\text{Im}(z)|\leq \frac{\gamma}{2}\}.
$$
Setting $z=t+is$ this set can be expressed as
$$
A^{\complement} = \{\la_1+t^2-s^2 +2its:s^2 \leq \frac{\gamma^2}{4}, t \in \mathbb{R} \}
$$
and  setting $y=2ts$ we get

\begin{equation}\label{Acomplement}
    \begin{split}
        A^{\complement} &= \{\la_1+ \frac{y^2}{4s^2}-s^2 +iy:s^2 \leq \frac{\gamma^2}{4}, y \in \mathbb{R} \}\\
       &= \{   x+iy : x\geq \la_1 - \frac{\gamma^2}{4} +\frac{y^2}{\gamma^2}, y \in \mathbb{R} \}.
    \end{split}
\end{equation}
On the other hand, by Lemma \ref{lemma2},
$$
Q_{1,k}=\{a_0\left(\frac{n-1}{2}-k\right)^2+z^2: |\text{Im} z|\leq \frac{\sqrt{a_0}(n-1)}{2} \}
$$
which can similarly be expressed as
$$
Q_{1,k}=\{x+iy: x \geq a_0\left(\frac{n-1}{2}-k\right)^2-\frac{a_0(n-1)^2}{4}+ \frac{y^2}{a_0(n-1)^2}, y \in \mathbb{R} \}.
$$
In other words both $A^{\complement}$ and $Q_{1,k}$ are parabolic regions to the right of parabolas of the type $x=x_0+by^2$. For $A^{\complement}$, the constant $b$ is $b_1=\frac{1}{\g^2}$, and for $Q_{1,k}$ it is $b_2=\frac{1}{a_0(n-1)^2}$. If $\gamma <(n-1)\sqrt{a_0}$ then $b_1>b_2$ and hence the parabola of $A^{\complement}$ is strictly contained in $Q_{1,k}$ when $x+iy$ has $x$ large enough. This gives us the contradiction. As a result we must have $\gamma=(n-1)\sqrt{a_0}$.
\end{proof}
Now that we have found the exponential rate of volume growth $\g$ of $M$, we are in position to determine the $L^p$ spectrum of the Laplacian on $k$-forms, if we make the assumption that the the bottom of the $L^2$ spectrum of the Laplacian on $k$-forms is $\la_1=a_0(\frac{n-1}{2}-k)^2$. 
\begin{prop} \label{Spectrum}
   Let $M$ be a warped product at infinity where the warping function $f \in B$. Assume that the bottom of the $L^2$ spectrum of the Laplacian on $k$-forms is 
    $$
    \la_1=a_0(\frac{n-1}{2}-k)^2,
    $$
    then  $\sigma(p,k,\D)= Q_{p,k}$.
    \end{prop}
\begin{proof}
Let $A$ as in (\ref{eq55}). For $\gamma=(n-1)\sqrt{a_0}$  and $\la_1= a_0(\frac{n-1}{2}-k)^2$ we observe that $$A^{\complement}=Q_{p,k}.$$
By Theorem \ref{thm2} and the proof of Proposition \ref{propprop} we have $$\sigma(p,k,\D) \subset A^{\complement}.$$
At the same time, by Theorem \ref{thm1} we have
$$
Q_{p,k} \subset \sigma(p,k,\D),
$$
which give us the proposition.
\end{proof}
Below we will give some motivation of what the value of $\la_1$ in Proposition \ref{Spectrum} reflects.

We begin by looking at a particular class of warped product at infinity which are conformally compact. Lets recall the definition of a conformally compact manifold.
\begin{Definition} \cite{borthwick2001scattering}
Let $X$ be a smooth manifold with boundary $\partial X$, equipped with an arbitrary smooth metric $\bar{g}$. A boundary-defining function on $X$ is a function $x\geq 0$ such that $\partial X=\{x=0\}$ and $dx\neq 0$ on $\partial X$. A conformally compact metric on the interior of $X$ is a metric of the form 
$$
g=\frac{\bar{g}}{x^2}.
$$
\end{Definition}
Borthwick proves the following structure theorem for conformally compact manifolds.
\begin{prop}[Proposition 3.1 \cite{borthwick2001scattering}] \label{Proposition Borthwick}
    Let $X$ be a compact manifold with $g$ a conformally compact metric. Then, there exists a product decomposition $(x,y)$ near $\partial X$ such that 
    \begin{equation} \label{conforma metric}
        g=\frac{dx^2}{a(y)^2 x^2}+ \frac{h(x,y,dy)}{x^2} +O(x^{\infty})
    \end{equation}
Here $-a(y)^2$ is the limiting curvature at infinity.
\end{prop}

\begin{lemma} \label{lemma}
 Let $M$ be a warped product at infinity where the warping function $f \in B$ is restricted to satisfy $f(r)\sim ce^{\sqrt{a_0}r}$, for some $a_0>0$,  as $r \to \infty$. Then the metric $g=dr^2+f^2(r)g_N$, on $M\setminus K=(c_0,\infty)\times N$ is a conformally compact metric, with limiting curvature $-a_0$ at infinity.
\end{lemma}

\begin{proof}
  By Proposition \ref{Proposition Borthwick} it suffices to show that the metric $g$ can be expressed as 
$$
 g=\frac{dx^2}{a_0 x^2}+ \frac{h(x,y,dy)}{x^2} +O(x^{\infty}).
$$
Setting
$$
r=-\frac{lnx}{\sqrt{a_0}} \Leftrightarrow x=e^{-\sqrt{a_0}r}
$$
we get $dr^2=\frac{1}{a_0x^2}dx^2$.
Therefore, the metric $g$ can be rewritten as: 
\begin{equation*}
        \begin{split}
            g & =dr^2+f(r)^2 g_N \\
              &=\frac{1}{x^2 a_0} dx^2 +\frac{[f(\frac{-ln x}{\sqrt{a_0}})x]^2}{x^2}g_N
        \end{split}
    \end{equation*}
Since $f(r)\sim ce^{\sqrt{a_0}r}$, as $r \to \infty$ we have 
\begin{equation*}
        \begin{split}
            c & =\lim_{r\to \infty} f(r) e^{-\sqrt{a_0} r} \\
              &=\lim_{x \to 0} f\left(\frac{-lnx}{\sqrt{a_0}}\right)x
        \end{split}
    \end{equation*}
Thus, if we set 
$$
h=f^2\left(\frac{-lnx}{\sqrt{a_0}}\right) x^2 g_N
$$
we get 
$$
h\to c\; g_N \; \text{as}\; x\to0
$$
which tells us that $M$ is a conformally compact manifold with boundary $\partial X=N$ at infinity.  
\end{proof}

\begin{theorem} [(1.3) Theorem \cite{mazzeo1988hodge}]
    For the conformally compact metric g in (\ref{conforma metric}), if  $-a_0$ is the maximum limiting curvature at infinity for some $a_0>0$, then the essential spectrum of $\D$ the Laplacian on $L^2$-integrable $k$-forms is $[a_0 \frac{(n-2k-1)^2}{4}, \infty)$, $\{0\} \cup [\frac{a_0}{4}, \infty)$, $[a_0 \frac{(n-2k+1)^2}{4}, \infty)$  for $k<\frac{n}{2}$, $k=\frac{n}{2}$, $k>\frac{n}{2}$ respectively.
\end{theorem}
By this Theorem and Lemma \ref{lemma} we have an immediate result for the bottom of the essential spectrum of $M$ (in our case $-a_0$ is the maximum limiting curvature at infinity).
\begin{prop}
 Let $M$ be a warped product at infinity where the warping function $f \in B$ is restricted to satisfy $f(r)\sim ce^{\sqrt{a_0}r}$, as $r \to \infty$. Then,  the $L^2$ essential spectrum of the Laplacian on forms on $M$ is $[a_0 \frac{(n-2k-1)^2}{4}, \infty)$ for $k<\frac{n}{2}$,  $\{0\} \cup [\frac{a_0}{4}, \infty)$, $[a_0 \frac{(n-2k+1)^2}{4}, \infty)$  for $k<\frac{n}{2}$, $k=\frac{n}{2}$, $k>\frac{n}{2}$ respectively.
\end{prop}
In other words, the $L^2$ spectrum of the Laplacian on $k$-forms over such a manifold, consists of isolated eigenvalues in the interval $[0, a_0 \frac{(n-2k-1)^2}{4})$ together with the interval $[a_0 \frac{(n-2k-1)^2}{4}, \infty)$. The bottom of the essential spectrum $\la_1=[a_0 \frac{(n-2k-1)^2}{4}]$ coincides with the vertex of the parabola $Q_{p,k}$ in Lemma \ref{lemma2}. 

Interpreting the above warped products as conformally compact manifolds, allowed us to obtain a simple proof for what their essential spectrum should be. Compare for example with the more intricate arguments in \cite{antoci2004spectrum}.

By the above Proposition, if we have a warped product at infinity $M$ with warping function $f \in B$ such that $f(r)\sim ce^{\sqrt{a_0}r}$, as $r \to \infty$, and we assume that the spectrum of the Laplacian on $L^2$-integrable $k$-forms does not have isolated eigenvalues, we have that $\sigma(k,2,\D)=[a_0 \frac{(n-2k+1)^2}{4}, \infty)$. Thus, $M$ satisfies the assumption of Proposition \ref{Spectrum} and we immediately get the following result.
\begin{theorem}
    \label{ResultA}
Let $M$ be a warped product at infinity where the warping function $f \in B$ is restricted to satisfy $f(r)\sim ce^{\sqrt{a_0}r}$, for some $a_0>0$, as $r \to \infty$. Let $k\neq \frac{n}{2}$ and assume that the $L^2$ spectrum of the Laplacian on $k$-forms over $M$ has no isolated eigenvalues of finite multiplicity. Then the $L^p$ spectrum of the Laplacian on forms on $M$ is $Q_{p,k}$.
\end{theorem}
In the particular case when $M$ is in addition an Einstein manifold we have a more precise result about its $L^p$ spectrum.
\begin{Corollary} \label{resultB}
    Suppose that $M$ is a warped product of negative curvature, which is in addition an Einstein manifold, and such that the Yamabe invariant of $N$ is non-negative. Then the $L^p$ spectrum of the Laplacian on functions $\sigma(0,p,\D)$ is precisely $Q_{p,0}$, with $-a_0$ the curvature of the warped product at infinity.
\end{Corollary}
\begin{proof}
Let us note the following two results, that will give us the proof of the corollary. In \cite{lee1994spectrum} Lee, show that the $L^2$ spectrum of the Laplacian on $M$ as in our assumptions has no isolated eigenvalues if $M$ is Einstein. From \cite{besse2007einstein}  9.109, 9.110 we have that $M$ is Einstein with negative curvature if and only if the warping function $f(r)$ is equal to the one of the following three functions: $cohs (\sqrt{a_0}r), \; e^{\sqrt{a_0}r},  \; sinh(\sqrt{a_0}r)$, for some $a_0>0$. Since all three of these functions belong to $B$, the corollary follows from Theorem \ref{ResultA}.
\end{proof}

We end this section by exhibit a large class of functions belonging to the set $B$. This class can be characterized by the asymptotic solutions of a certain Riccati differential equation. We have

\begin{lemma}
\label{lemma1}
Let $q:(a, \infty) \to \mathbb{R}$ be  monotonic and 
$$
\int_{T_0}^{\infty} |q(s)| ds < \infty,
$$
for some $T_0>a$. Let $ f \in C^2(a,\infty)$ be a solution of the differential equation $f''-(a_0+q)f=0$ with $a_0 > 0$. Then,
$$
\frac{f''}{f}=a_0 +o(1), \; \; \left(\frac{f'}{f}\right)^2=a_0+o(1), \text{as}  \; t \rightarrow \infty.$$
\end{lemma}
Let us note that the assumption $
\int_{T_0}^{\infty} |q(s)| ds < \infty$
 together with that $q$ is monotonic implies that as $q \to 0$ as $t \to \infty$.
An example of a function satisfying the assumptions of Lemma \ref{lemma1} is $f(r)= c sinh(\sqrt{a_0}r)$, $a_0>0$, $c \in \mathbb{R}$ with $q=0$. 

The proof of Lemma \ref{lemma1} is based on the following.
\begin{prop}[Exercise 9.9 (a) \cite{hartman2002ordinary}]
    Let $\la >0$ and $q(t)$ be a continuous complex-valued function for large $t$ such that
    \begin{equation}
        \label{eq1Hartman}
        Q_\la(t)=\int_t^\infty q(s) e^{-2\la s}ds \; \text{exists},
    \end{equation}
       \begin{equation}
        \label{eq2Hartman}
        \int^\infty Q_\la (t) e^{2\la t}dt \; \text{exists},
    \end{equation}
    and
    \begin{equation}
        \label{eq3Hartman}
        \int^\infty |Q_\la (t)|^2 e^{4\la t}dt \; \text{exists}.
    \end{equation}
    Then $u''-(\la^2+q(t))u=0$ has a pair of solutions satisfying, 
    $$
  u\sim e^{\pm \la t}, \; \;   \frac{u'}{u}=\pm \la+ e^{2\la t} Q_\la(t) +o(1), \text{as} \; t\to \infty.
    $$
\end{prop}
\textit{Proof of Lemma \ref{lemma1}}

By the above Proposition it suffices to show that the function $q:(a, \infty) \to \mathbb{R}$ satisfies the assumptions (\ref{eq1Hartman}), (\ref{eq2Hartman}), (\ref{eq3Hartman}) and $e^{2 \la t} Q_\la(t) \to 0$  as $t \to \infty$.
By the monotonicity of $q$ we have
\begin{equation}
    \label{eqmonotonicity}
    \begin{split}
        |Q_\la(t)| \leq \int_t^\infty |q(s)| e^{-2 \la s} ds \leq & |q(t)| \int_t^\infty e^{-2\la s} ds \\
        =& |q(t)|\frac{1}{2\la} e^{-2\la t} < \infty,
    \end{split}
\end{equation}
for every $t$. Now (\ref{eqmonotonicity}) gives 
\begin{equation}
    \label{eqmonoto2}
    \int_{t_0}^\infty |Q_\la (t)| e^{2 \la t} dt \leq c \int_{t_0}^\infty |q(t)| dt < \infty
\end{equation}
and 
\begin{equation*}
    \begin{split}
        \int_{t_0}^\infty |Q_\la (t)|^2 e^{4 \la t} dt \leq& \int_{t_0}^\infty |q (t)|^2 e^{-4 \la t} e^{4 \la t}dt  \\
            \leq  & \int_{t_0}^\infty c |q (t)|^2 dt \leq c +\int_{T_0}^\infty |q(t)|dt.
    \end{split}
\end{equation*}
The last inequality follows, since in a neighborhood of $\infty$ we have $|q|<1$, which gives $|q|^2\leq |q|$. Furthermore, 
(\ref{eqmonotonicity}) also gives $Q_\la (t) e^{2 \la t } \to 0$ as $t \to \infty$, since it implies
$$
|e^{2\la t} Q_\la(t)| \leq c |q(t)|,
$$
for every $t$ and by our assumption  $q \rightarrow 0$ as $t \rightarrow \infty$.

\section{The $L^p$ spectrum on a class of Kleinian Groups}

In \cite{davies1988lp} Davies, Simon and Taylor studied the $L^p$ spectrum of the Laplace-Beltrami operator $\D_\G$ on non compact  quotients $M=\mathbb{H}^{N+1}/\Gamma $, where $\G$ is a geometrically finite group. Under the additional assumptions that $M$ is either of finite volume or cusp-free, they determine explicitly the $L^p$-spectrum of $\D_\G$ for $1 \leq p \leq \infty$. More precisely they proved.

\begin{theorem} [\cite{davies1988lp} Theorem 8,9] \label{Theorem 5 Intro}
    Let $\{E_0, \dots, E_m\}$ be a finite set of eigenvalues of $\D_{\G,2}$ such that $E_j < \frac{N^2}{4}$. Then, if  $M$ has no cusps or has finite volume then 
$$
    \sigma(p,0,\D_\G)=\{E_0, \dots, E_m\} \cup Q_p,
    $$ where $Q_p$ is a parabolic region in the complex plane.
\end{theorem}
In this section our main goal is to generalize the above Theorem for the Laplacian on forms $\vec \D_\G$  in the case where $M$ has no cusps. In order to do this we define and study the Laplacian on forms $\vec \D_\G$ on quotient spaces $M=\mathbb{H}^{N+1}/\Gamma $. The main idea of the proof and which is based on the argument in \cite{davies1988lp}, is to split the Laplacian on forms into two operators, one corresponding to the span of eigenforms with eigenvalues in the discrete isolated spectrum, and the second one acting on the quotient.  Let us note that for simplicity, throughout this section where it is clear from the context, we denote the Laplacian on forms $\vec \D_\G$ by  $\D_\G$ and $L^p$ to denote $L^p(\Lambda^k(M))$. Also, in this section, we will use $n=N+1$ for the dimension of hyperbolic space to coincide with other literature.

Let $f: \mathbb{H}^{N+1} \to \mathbb{C}$, $f$ is called $\Gamma$-invariant if $f(\g x)=f(x)$ for all $\g \in \G$ and $x\in \mathbb{H}^{N+1}$. 
Denote by $\D$ the Laplacian on $\mathbb{H}^{N+1}$ and by $\D_\G$ the Laplacian on $\mathbb{H}^{N+1} / \G $. The Laplacian $\D_\G$ is well defined due to the $\G$-invariance of the Laplacian under the isometries, $\D \g ^*=\g^* \D$, $\g \in$ Isom$(\mathbb{H}^{N+1})$, here $\g^* f = f\circ \g$ is the pull back.

Now we see how the above extends to $k$-forms. A $k$-form $\omega$ is $\G$-invariant if $\g^*\omega=\omega$. Here the pull back on forms is defined by
\begin{align*}
 \g^*:& \Lambda^k(M)\to \Lambda^k(M)   \\
  (\g^*(\omega))_x &=\omega_{\g (x)} \circ (d\g_x\times \cdots \times d\g_x ),\\
\end{align*}
(see \cite{tu2011manifolds}). Also we denote by $\D$ the Laplacian on forms on $\mathbb{H}^{N+1}$
and by $\D_\Gamma$ the Laplacian on forms on $\mathbb{H}^{N+1} / \G $. Similarly, $\D_\G$ is well-defined due to the $\G$-invariance of $\D$ under isometries, $\D \g ^*=\g^* \D$, $\g \in$ Isom$(\mathbb{H}^{N+1})$ (see \cite{craioveanu2013old} Proposition 2.11).
For any $\G$-invarinat form $\omega$ on $\mathbb{H}^{N+1}$ we define the induced form 
$$
\tilde{\omega} :  \mathbb{H}^{N+1} / \G  \to  \mathbb{C}
$$ 
by 
$$
\tilde{\omega} (\G x)=\omega(x),
$$
where $\G x$ is the orbit of $x$. If we set
$$
\mathcal{A}_\G(\mathbb{H}^{N+1}) =\{ \text{all $\G $-invariant forms} \; \omega \;\text{on} \;\mathbb{H}^{N+1} \}
$$
this induces the map 
\begin{align*}
  T:\mathcal{A}_\G(\mathbb{H}^{N+1})& \longrightarrow \mathcal{A}(\mathbb{H}^{N+1}/\G) \\
  \omega &\longmapsto\; \;\tilde{\omega}. \\
\end{align*}
Note that the following relation between $T$ and $q: \mathbb{H}^{N+1} \to \mathbb{H}^{N+1} /\G $ holds
$$
T(\omega) (q(x))=\tilde{\omega}(\G x)=  \omega( x).
$$
It is well-known that the relationship 
 $$
    e^{-t\D_\G}(T f)= Te^{-t\D }f,
    $$
holds for the heat semigroups corresponding to $\D$ and $\D_\G$ on $\G$-invariant functions $f$ over $\mathbb{H}^{N+1}$ (see Corollary 3 in \cite{davies1988heat}, Lemma 2.14 in \cite{weber2007heat}). A similar relationship also holds for $\G$-invariant forms. 
\begin{prop}\label{Corrolary 3}
    Let $\omega$ be a continuous $\G$-invariant $k$-form on $\mathbb{H}^{N+1}$ such that the restriction of $\omega$ to a fundamental domain $F \subset \mathbb{H}^{N+1}$ for $\G$ has compact support, then 
    $$
    e^{-t\D_\G}(T\omega)= Te^{-t\D }\omega.
    $$
\end{prop}
\begin{proof}
    The proposition is based on the uniqueness, of solutions as stated in Corollary \ref{uniqueness of Cauchy on forms} of the heat equation on forms. Let $\omega$ be a continuous $\G$-invariant $k$-form on $\mathbb{H}^{N+1}$ such that the restriction of $\omega$ to a fundamental domain $F \subset \mathbb{H}^{N+1}$ for $\G$ has compact support. We define $\omega_t(x)$, $\eta_t(x)$ to be solutions of the Cauchy problems
\begin{equation} \label{heat on D}
    \begin{cases}
     ( \partial_t +\Delta) \omega_t(x)=0 \\
     \;\omega_0(x)=\omega(x)
    \end{cases} 
\end{equation}
on $\mathbb{H}^{N+1}$, and 
    \begin{equation} \label{heat on D G}
    \begin{cases}
     ( \partial_t +\Delta_\G) \eta_t(\tilde{x})=0 \\
     \;\eta_0(\tilde{x})= T \omega(\tilde{x})
    \end{cases} 
\end{equation}
respectively on $\mathbb{H}^{N+1}/ \G$.
The solutions are given by  
\begin{equation} \label{Caucy D}
    \omega_t(x)= (e^{-t\D } \omega) (x)
\end{equation}
and
\begin{equation} \label{Caucy D G}
    \eta_t(\tilde{x})= (e^{-t\D_\G } T \omega) (\tilde{x}).
\end{equation}
So, if $\omega_t(x)$ were $\G$-invariant, then, by applying $T$ on both sides of (\ref{Caucy D}) would have 
\begin{equation}\label{eq10}
    \omega_t(x)=(T\omega_t)\circ q(x)=T(e^{-t\D } \omega) \circ q(x).
\end{equation}
Moreover, if 
\begin{equation}\label{eq11}
    \omega_t(x)=\eta_t \circ q(x)
\end{equation}
holds, in other words if $\eta_t \circ q(x)$ solves (\ref{heat on D}), then by the uniqueness of solutions 
\begin{equation}\label{eq12}
    \eta_t \circ q(x)=(e^{-t\D_\G}T \omega)(q(x))=\omega_t(x)= T(e^{-t\D } \omega) \circ q(x).
\end{equation}
In other words, we get the desired equality
$$
    e^{-t\D_\G}(T\omega)= Te^{-t\D }\omega.
$$
So it is enough to prove  that the solution $\omega_t(x)$ is $\G$-invariant and $\omega_t(x)=\eta_t \circ q(x)$.
Let us begin by showing that  $\omega_t(x)$ is $\G$-invariant. By the definition of the heat kernel 
$$
\omega_t(x)=\int_{\mathbb{H}^{N+1}}<\vec{p}(t,x,y),\omega(y)>dy,
$$
where $\vec{p}(t,x,y)$ is the heat kernel on forms on $\mathbb{H}^{N+1}$.
We compute 
\begin{align*}
 \g^*\omega_t(x)=\omega_t(\g x)&=\int_{\mathbb{H}^{N+1}}<\vec{p}(t,\g x,y),\omega(y)>dy  \\
  &= \int_{\mathbb{H}^{N+1}}<\vec{p}(t, x,\g^{-1}y),\omega(y)>dy \\
  &= \int_{\mathbb{H}^{N+1}}<\vec{p}(t,x,y),\omega(\g y)>dy \\
  &= \int_{\mathbb{H}^{N+1}}<\vec{p}(t,x,y),\omega(y)>dy \\
\end{align*}
where in the second equality we have used Proposition \ref{heat isometry}, in the third equality we have used Proposition \ref{heat invariance} and the last equality follows by the $\G$-invariance of $\omega$.

Finally, by the $\G$-invariance of $\D$ and that $q$ is a Riemannian covering we get
$$
\D (\eta_t \circ q(x))= (\D_\G \; \eta_t) \circ q(x).
$$
Combining this with (\ref{heat on D G}) we get 
$$
- \D(\eta_t \circ q(x))=\partial_t(\eta_t \circ q(x))
$$
Thus, $\eta_t \circ q(x)$ is a solution to the heat equation on $\mathbb{H}^{N+1}$. Here we use the fact that since $\omega$ is $\G$-invariant it satisfies the initial condition also. The proposition follows by the uniqueness of  (\ref{heat on D G}).
\end{proof}

On $L^p$ integrable $k$-forms $\D_{\G,p}$ is defined via its semigroup as in the previous section.  For $p=2$ we will use the notation $\D_\G$ for simplicity. Keeping similar notation with the previous Section we will denote by $P'_{p,k}$
the parabolic curve

$$
P'_{p,k}=\{-(\frac{N}{p}-k+is)[N(\frac{1}{p}-1)+k+is], s \in \mathbb{R}\},
$$
and the parabolic region to the right of the curve $P'_{p,k}$  by  
 \begin{equation} 
     Q'_{p,k}=\{\left(\frac{N}{2}-k\right)^2+z^2: |\text{Im} z|\leq N|\frac{1}{p}-\frac{1}{2}| \}.
 \end{equation}

These are the parabolas from the previous Section which corresponds to a manifold of dimension $n=N+1$ and with limiting curvature at infinity $-a_0=-1$, as is the case of $\mathbb{H}^{N+1}$.

Mazzeo and Phillips \cite{mazzeo1990hodge} computed the $L^2$ spectrum of Laplacian $\D_\G$ for quotients $M=\mathbb{H}^{N+1}/ \G$, where $\G$ is a geometrically finite group,  proving the following result.
\begin{theorem}[\cite{mazzeo1990hodge} Theorem 1.11]\label{Mazzeo Thm}
When $M$ is not compact the essential spectrum consists of the entire interval $[(N/2-k)^2, \infty)$, when $k\leq N/2$ and  $[(N/2-k+1)^2, \infty)$, when $k\geq N/2$. The only exception is when $k=\frac{n}{2}$ and $M$ has infinite volume. In that case in addition to the above, $0$ is an eigenvalue of infinite multiplicity.
\end{theorem}
By definition every point outside the essential spectrum in an isolated eigenvalue of finite multiplicity. Lax and Phillips \cite{lax1982asymptotic} showed that for the Laplacian on functions for quotients $M=\mathbb{H}^{N+1}/ \G$, where $\G$ is a geometrically finite group, $\sigma(0,2,\D_\G) \setminus \sigma_\text{ess}(0,2,\D_\G)$ must be a finite set of eigenvalues of finite multiplicity. We expect that a similar result must hold for the Laplacian on forms, but this is still an open problem. We will see that if we make this additional assumption on the set of isolated eigenvalues of finite multiplicity for the manifold we can precisely compute the $L^p$-spectrum of the Laplacian on $k$-forms.

In this section we compute the $L^p$-spectrum of the Laplacian on $k$-forms. We will prove the following result.

Our proof extends the method defined in \cite{davies1988lp}, and relies on various properties of the resolvent operator and the heat kernel which we develop below. We will prove the result for $0\leq k \leq \frac{N}{2}$, and for $k> \frac{N+1}{2}$ it will follow by Poincare duality. 

Here we will be a little bit more precise and keep the sub index $\D_p$ for the Laplacian on $L^p(\Lambda^k(\mathbb{H}^{N+1}))$ and write $\D_{\G,p}$ for the Laplacian on $L^p(\Lambda^k(M))$.

\begin{lemma}
\label{lemma3}
    Let $0\leq k\leq N/2$. Whenever $|Imz|>N/2$, the resolvent operator on forms $(\D_1 - (N/2-k)^2 -z^2)^{-1}$ is bounded on $L^1(\Lambda^k(\mathbb{H}^{N+1}))$.
\end{lemma}
  \begin{proof}
      This is obtained immediately from Theorem \ref{thm2} as shown in \cite{charalambous2024lpspectral}, since the Hyperbolic space  $\mathbb{H}^{N+1}$ has constant negative curvature $-1$ and rate of volume growth $\g=N$. Thus $|Imz|>\g|1/p-1/2|$ reduces to $|Imz|>N/2$.  
  \end{proof}

\begin{prop}\label{Corollary 33}
    Let $0\leq k\leq N/2$. Whenever $|Imz|>N/2$, the resolvent operator $(\D_{\G,1} - (N/2-k)^2 -z^2)^{-1}$ is bounded on $L^1(\Lambda^k (M))$.
\end{prop}
\begin{proof}
The map $T: L^1(\mathbb{H}^{N+1}) \to L^1(M)$ has an adjoint which is an one to one map from $L^\infty (M) \to L^\infty (\mathbb{H}^{N+1})$. As a result,  T is an onto map from $L^1(\mathbb{H}^{N+1}) \to L^1(M) $.  For any $\la \notin \sigma(k,1,\D_{\G,1}) $ we have the functional analytic formula
$$(\D_{\G,1} -\la)^{-1}= \int_0^\infty e^{-t\D_{\G,1}} e^{\la t} dt,
    $$
    and a similar formula for $(\D_1-\la)^{-1}$ for any $\la \notin \sigma(k,1,\D) $. By Proposition \ref{Corrolary 3} after restricting the heat kernel on $L^1$, we get 
    $$
    (\D_{\G,1} -\la)^{-1} Tw=T (\D_{1} -\la)^{-1} w
    $$
    for any $\la \notin \sigma(k,1,\D) $. Using the fact that 
    $$
    ||(\D_{\G,1} -\la)^{-1}Tw || \leq ||(\D_{\G,1} -\la)^{-1}|| \;||Tw||
    $$
   gives 
    $$
    ||(\D_{\G,1}-\la)^{-1}||\leq ||(\D_{1}-\la)^{-1}||
    $$ for all $\la \notin \sigma(k,1,\D) $. The proposition follows by combining this with Lemma \ref{lemma3}.
\end{proof}
We will now prove that the heat operator corresponding to $\D_\G$ is bounded from $L^2$ to $L^\infty$, it is in other words ultracontractive.

\begin{prop}\label{Proposition 4}
    There exists $c(t)=c>0$ such that 
$$
||e^{-\D_\G t} \omega||_\infty \leq c ||\omega||_2, \; \text{for every} \;\omega \in L^2(\Lambda^k(M)).
$$
\end{prop}
\begin{proof}
The proof is similar to Proposition 4 in \cite{davies1988lp}. Firstly, we will show that $e^{-t\D_\G}: L^1(\Lambda^k(M)) \to L^\infty (\Lambda^k(M))$ is bounded. Let us note that 
\begin{equation} \label{bounded}
   ||e^{-\D_\G t} \omega||_\infty \leq  \left\Vert \int_M | \vec{p}_\G(t,z,w)| \;|\omega(w)| dw \right\Vert_\infty,
\end{equation}
hence it suffices to show  that $|\vec{p}_\G (t,z,w)|$ is uniformly bounded with respect to $z,w$. 
Since on $ \mathbb{H}^{N+1} / \G$ the Weitzenbock tensor on $k$-forms is bounded below by a negative constant $-K_2$, by Proposition \ref{Weitzenbock tensor} we have 
\begin{equation}
    \label{eq20}
    |\vec{p}_\G (t, z, w) | \leq e^{tK_2}  |p_\G (t, z, w) |,
\end{equation}
where $p_\G$ is the heat kernel on functions over $ \mathbb{H}^{N+1} / \G$. So it suffices to estimate the heat kernel on functions $p_\G (t, z, w)$. This computed in Proposition 4 in  \cite{davies1988lp}. For the sake of completeness we work out the details. If $0<t\leq 1$, then the heat kernel on functions for $ \mathbb{H}^{N+1} $ satisfies the estimate 
$$
0<p(t,z,w)\leq c_0t^{-\frac{N+1}{2}}e^{-\frac{N\rho}{2}}e^{-\frac{\rho^2}{4t}}(1+\rho)^\frac{N}{2}
$$
as shown in \cite{davies1988heat}, where $\rho$ is the hyperbolic distance from $z$ to $w$. Thus, 
$$
0<p_\G(t,z,w)\leq c_1 t^{-\frac{N+1}{2}} \sum_{\g \in \G} e^{-\frac{\rho(z,\g z)^2}{4t}}.
$$
Since $M$ has no cusps, the sum is bounded independently of $z$ by \cite{davies1988heat}, \cite{MR0450547}.
So, the above with the well-known property of the heat kernel
$$
p(t,z,w) \leq \sqrt{p(t,z,z)p(t,w,w)},
$$
(see Exercise 7.21 in \cite{grigoryan2009heat}) gives
$$
0<p_\G(t,z,w)\leq c t^{-\frac{N+1}{2}}.
$$
Now, we show
$$
e^{-t\D_\G} : L^\infty(\Lambda^k(M)) \to L^\infty (\Lambda^k(M))
$$
is bounded. From (\ref{bounded}) it suffices to show 
$$
\sup_z \int_M  |\vec{p}_\G (t, z, w)| dw \leq c(t),
$$
where $c(t)$ is a uniform constant depending on $t$. By (\ref{eq20}) and 
$$
\sup_z \int_M  |p_\G (t, z, w)| dw \leq 1,
$$
we get 
$$
\sup_z \int_M  |\vec{p}_\G (t, z, w)| dw \leq \sup_z \int_M  e^{K_2t} |p_\G (t, z, w)| dw \leq c(t),
$$
where $c(t)$ is any constant depending on $t$. Since $e^{-t\D_\G}$ is bounded from $L^1(\Lambda^k(M))$ to $L^\infty(\Lambda^k(M))$ and from $L^\infty(\Lambda^k(M))$ to  $L^\infty (\Lambda^k(M))$, using interpolation (see \cite{davies1989heat} p.3 ) we get that 
$$
e^{-t\D_\G} : L^2 (\Lambda^k(M)) \to L^\infty (\Lambda^k(M))
$$
is also bounded.
\end{proof}

\begin{Corollary}\label{Collorary 5}
    If $\omega$ is an $L^2$ $k$-eigenform for the Laplacian on $ \mathbb{H}^{N+1} / \G$, then  $\omega \in L^\infty (\Lambda^k(M))$.
\end{Corollary}
\begin{proof}
Suppose that $\omega \in L^2$ be an eigenform of $\D_\G$ with eigenvalue $E$. Then $\omega$ satisfies  $e^{-t\D_\G} \omega= e^{-tE} \omega$ for every $t>0$. 
For $t=1$, and using Proposition \ref{Proposition 4} we have
$$
||e^{-E} \omega||_\infty=||e^{-\D_\G} \omega||_\infty \leq c ||\omega||_2.
$$ Hence
$$
||\omega||_\infty \leq e^E c || \omega||_2,
$$
which gives the corollary.
\end{proof}
Let us note that if $\mathbb{H}^{N+1} / \G$ has cusps but it is of infinite volume, then as noted in \cite{davies1988lp}, Fourier analysis shows that $\p_0$ the $L^2$ eigenfunction corresponding to the first eigenvalue, diverges to $\infty$ in each cusp. Although this is not known, it is expected that this would also happen for eigenforms. Our assumption that $\mathbb{H}^{N+1} / \G$ is cusp-free, seems to therefore be necessary to conclude that every $L^2$- eigenform must also belong to $L^p$ for all $p \in [2,\infty]$. We will now move to address the case $p \in [1,2]$.

In the following Lemma we will compute one inclusion which the spectrum of $\D_\G$ on $L^p (\Lambda^k(M))$ satisfies.

\begin{lemma} \label{lemma6}
    If $1\leq p \leq 2$, then 
    $$
    \sigma(k,p,\D_{\G}) \subset \{E_0,\dots,E_m\} \cup Q'_{p,k}.
    $$
\end{lemma}
\begin{proof}
Denote by $\p_r$ the $L^2$-eigenform corresponding to the eigenvalue $E_r$. By Corollary   \ref{Collorary 5}  $\p_r \in L^\infty$. Thus, using interpolation we also have that $\p_r \in L^q$ for all $q\geq 2$. 

Let $1\leq p \leq 2$. Since $\p_r \in L^q$ for every $q \geq 2$, we also have that $\p_r \in L^{p^*}$ for $\frac{1}{p}+\frac{1}{p^*}=1$. So for any $\omega \in L^p$ the $L^p-L^{p^*}$ pairing $$(\omega, \p_r)=\int <\omega,\p_r>$$ is well defined and gives an operator $\tilde{\phi}_r \in L^{p^*}$ such that $\tilde{\phi}_r(\omega)=(\omega, \p_r)$. Define the subspace $L^p_1$ of $L^p$ by 
    $$
    L^p_1(\Lambda^k(M))=\{\omega \in L^p(\Lambda^k(M)) : \int <\omega,\p_r> =0, \text{for all} \; r \in \{0,\dots, m\}\}.
    $$
    Since $L^p_1$ is a closed subspace of $L^p$, we have that $L^2_1$ is a Hilbert space. 
    We will show that $L^p_1$ is invariant under $e^{-t\D_\G }$. Let $\omega \in L^p_1$. Then $e^{-t\D_\G} \omega \in L^p$ since the heat operator is bounded on $L^p$. Since $\tilde{\phi}_r \in L^{p^*}$,
\begin{equation*}
    \begin{split}
       (\phi_r, e^{-t\D_\G} \omega) &= \int <\p_r ,e^{-t\D_\G} \omega> \\
 &=\int <e^{-t\D_\G}  \p_r ,\omega> \\
  &=\int e^{-tE_r}  <\p_r, \omega>=0.
     \end{split}
\end{equation*}
for all $r$, where we have used that the heat operator on $L^{p^*}$ is the adjoint of the heat operator on $L^p$. Therefore $e^{-t\D_\G} \omega \in L^p_1$.    
    By the previous claim we have that,     $e^{-\D_\G t}\restriction$ is the subspace semigroup on $L^p_1$  and we define $\D_{p,1}$ to be its generator.  $\D_{p,1}$ is the restriction of $\D_{\G,p}\restriction_{L_1^p}$ with domain the intersection of the domain of $\D_{\G,p}$
with $L_1^p$. For $p=1$ we have $\D_{1,1}=\D_{\G,1} \restriction_{L_1^1}$. As a result, the bound of $\D_{\G,1}$ from Proposition  \ref{Corollary 33} holds on $\mathcal{D}(\D_{\G,1}) \cap L^1_1$. This gives that 
$$
(\D_{1,1}-(N/2-k)^2-z^2)^{-1} 
$$
is bounded on $L_1^1$ for $|\text{Im}(z)|>N/2$.  By replacing $z$ with $iz$ we get that 
$$
(\D_{1,1}-(N/2-k)^2+z^2)^{-1} 
$$
is bounded for $Rez>N/2$. 

Now we define the quotient semigroup $L^p_2=L^p(\Lambda^k(M))/L^p_1$ and denote its generator by $\D_{p,2}$. Let us note that dim$(L^p_2)=m+1$. This follows easily, if we set 
\begin{align*}
  T:L^p & \longrightarrow \mathbb{R}^{m+1} \\
  \omega &\longmapsto\; \;( \tilde{\phi}_0(\omega), \cdots, \tilde{\phi}_m(\omega)) \\
\end{align*}
and notice that $\text{Ker}T= L_1^p$ and $\text{dim} (\text{Im} T)=m+1$.
As a result 
$$
\sigma(k,p,\D_{p,2})=\{E_0,\dots,E_m\}.
$$

 We will now demonstrate that the spectra of 
$\D_{\G,p}$, $\D_{p,1}$ $\D_{p,2}$ are related in the following way
\begin{equation}\label{eq21}
  \sigma(k,p,\D_{\G,p})=\sigma(k,p,\D_{p,1}) \cup \sigma(k,p,\D_{p,2}).  
\end{equation}
 To show this, first observe that since $L_2^p$ is finite dimensional, there exists an isomorphism  $L^p(\Lambda^k(M))=L_1^p \oplus L_2^p$.
 Define
 $$
 R_0: L^p(\Lambda^k(M)) \to L^p(\Lambda^k(M)) 
 $$
 
by $$
R_0=(\D_{\G,1} +1)^{-1},
$$
and observe that $R_0$ leaves $L^p_1$ invariant, just as the heat operator does. Now define $R_1$  to be the restriction of $R_0$ to $L_1^p$ and $R_2$ to be the induced operator on the quotient space $L^p_2$. The isomorphism  $L^p(\Lambda^k(M))=L_1^p \oplus L_2^p$ implies that 
 \begin{equation} \label{...}
    \sigma(R_0)=\sigma(R_1)\cup \sigma(R_2).  
 \end{equation}
Then using the spectral mapping theorem for generators of one-parameter contraction semigroups \cite{davies1988lp} we get that (\ref{...}) implies (\ref{eq21})

 To complete the proof of the lemma it remains to compute $\sigma(k,p,\D_{p,1})$. Theorem \ref{Mazzeo Thm}  and the definition of $\D_{2,1}$ give
 $$
 \sigma(k,2,\D_{2,1})=[(N/2-k)^2, \infty).
 $$
By Spectral Theorem since $H=\D_{2,1}-(\frac{N}{2}-k)^2$ is non-negative self-adjoint and its spectrum is contained in $[0,\infty)$ we get that
$$
(\D_{2,1}-(N/2-k)^2+z^2)^{-1} 
$$
is bounded on $L_1^2$ whenever $\text{Re}z>0$. 
As we have shown above,
$$
(\D_{1,1}-(N/2-k)^2+z^2)^{-1} 
$$
is bounded on $L^1_1$ for $Rez>\frac{N}{2}$.  We are now ready to show 
$$
\sigma(k,p,\D_{p,1})=Q'_{p,k}, \; \text{for} \;p \in [1,2]
$$
This is a standard interpolation argument which we include for the sake of completion. By the above we have 
\begin{equation}
    \label{eq15} \left( \D_{p,1} -(\frac{N}{2}-k )^2 +(x+iy)^2 \right)^{-1}
\end{equation}
is bounded on $L_1^2$ for $x>0$ and $y \in \mathbb{R}$ and 
\begin{equation}
    \label{eq16} \left( \D_{p,1} -(\frac{N}{2}-k )^2 +(x+iy)^2 \right)^{-1}
\end{equation}
is bounded on $L_1^1$ for $x>\frac{N}{2}$ and  $y \in \mathbb{R}$.
We fix $\ep>0$ and $\alpha \in \mathbb{R}$ and define the operator
\begin{equation} \label{eq17}
    T(x+iy)= \left(  \D_{p,1} - (\frac{N}{2}- k)^2 +\frac{N^2}{4} (x +\ep +iy +i\alpha )^2\right)^{-1}
\end{equation}
For $y \in \mathbb{R}$ since $\frac{N}{2} \ep >0$ and $\frac{N}{2}(y+\alpha) \in \mathbb{R}$, (\ref{eq15}) gives that $T(iy)$ is bounded on $L_1^2$ for every  $y \in \mathbb{R}$. Also, for $y \in \mathbb{R}$ since $\frac{N}{2}+ \ep >\frac{N}{2}$ and $\frac{N}{2}(y+\alpha) \in \mathbb{R}$, (\ref{eq16}) gives that $T(1+iy)$ is bounded on $L_1^1$ for every  $y \in \mathbb{R}$.
Now, we fix any $p \in (1,2)$ and define the unique $t \in (0,1)$ such that
\begin{equation}
    \label{eq18} \frac{1}{p}=t+\frac{1}{2}(1-t).
\end{equation}
The Stein Interpolation Theorem (see \cite{davies1989heat} p.3 ) with $p_0=2$, $p_1=1$ and setting $x=t$ and $y=0$ in (\ref{eq17}) gives that 
$$
T(t)= \left(  \D_{p,1} - (\frac{N}{2}- k)^2 +\frac{N^2}{4} (t +\ep +i\alpha )^2\right)^{-1}
$$
is bounded on $L_1^p$, $1<p<2$ . Since $\ep >0$ and $\alpha \in \mathbb{R}$ were arbitrary, we have that 
$$
 \left(  \D_{p,1} - (\frac{N}{2}- k)^2 +z^2\right)^{-1}
$$
is bounded in $L_1^p$,  for $1<p<2$  whenever $Rez> \frac{N}{2}t$.
Since $t=\frac{2}{p}-1$ we have that 
$$
 \left(  \D_{p,1} - (\frac{N}{2}- k)^2 +z^2\right)^{-1}
$$
is bounded in $L_1^p$, for $1<p<2$, whenever $\text{Re}z>(\frac{2}{p}-1)\frac{N}{2}=(\frac{1}{p}-\frac{1}{2})N$.
Replacing $z$ with $iz$ and $z$ with $-z$ we get that 
$$
\sigma(k,p,\D_{p,1}) \subset \left\{\left(\frac{N}{2}-k\right)^2+z^2 : |\text{Im}z | \leq \left(\frac{1}{p}-\frac{1}{2}\right)N \right\}=Q'_{p,k}.
$$
\end{proof}
Now in order to get the reverse inclusion from the one given in the above Lemma, we need the following result.

\begin{lemma}
    \label{Lemma7} If $1\leq p\leq q \leq 2$, then the curve $P'_{q,k}$ which is the boundary of $Q'_{p,k}$ is contained in $\sigma(k,p,\D_{\G,p})$. 
\end{lemma}
\begin{proof}
As we saw in the proof of Proposition \ref{Proposition 4} $e^{-\D_\G t} : L^p \to L^q$ is bounded for every $1 \leq p \leq q \leq \infty$. Following the same analytical argument as in Proposition 3.1 in \cite{hempel1986spectrum} we can show that for any $1 \leq p \leq q \leq 2$ 
$$
 \rho(\D_{\G,q}) \supset  \rho(\D_{\G,p}).
$$
As a result we have 
$$
\sigma(k,q,\D_{\G,q}) \subset \sigma(k,p,\D_{\G,p}),
$$
for any $1 \leq p \leq q \leq 2$.
As in the proof of Theorem \ref{thm1}, it again suffices to prove
$$
P'_{q,k} \subset \sigma(k,q,\D_{\G,q}).
$$
In other words, for any $\la \in P'_{q,k} $ and $\ep>0$ we have to construct approximate eigenforms $\omega \in \Lambda^k (\mathbb{H}^{N+1}/ \G)$ such that
$$
|| \D_\G - \la) \omega ||_q \leq  \ep || \omega||_q
$$
We let $\mathbb{H}^{N+1}= [0,\infty) \times S^N$ with metric $g=dr^2+(sinhr)^2d \sigma^2$. Since $\G$ is geometrically finite, and has infinite volume, there is a region $\Omega \subset S^N$ and $b>0$ such that $(b,\infty)$ is contained in a fundamental domain of $\mathbb{H}^{N+1}/ \G$.  

On this region we consider approximate eigenforms of the type 

$$
\omega =\p f^\m (\chi(\theta) \e_0) \wedge dr,
$$
with $f(r)=sinhr$, $\chi(\theta) \in C^\infty_0 (\Omega)$, $\p=\p(r) \in C^\infty_0((b, \infty))$, $\m \in \mathbb C$ and $\e_0=\; \text{closed} \; (k-1)-\text{eigenform on } \: S^N$  with $\D_S \e_0= \la_0 \e_0$, where $\D_S$ is the Laplacian on $S^N$.

The procedure is the same as in Theorem \ref{thm1}. Firstly, we have to compute the action of $\D$ on $
\omega =\p f^\m (\chi(\theta) \e_0) \wedge dr
$.
Let $\e_2 =\chi(\theta) \e_0$, and   $h(r)= \p f^\m$. Since $\e_2$ is no longer a closed $(k-1)$-eigenform of $\D_S$, we not can use Corollary \ref{Corollary 3.1}. However, by Proposition \ref{Delta expression} with $\omega_1=0$ and $\omega_2=h(r) \e_2$ we have
\begin{equation*}
    \begin{split}
        \D \omega=&+hf^{-2} (\D_S \;\e_2) \wedge dr +(-1)^k 2hf'f^{-1} d_S \e_2 \\
       & -[h''+(N-2k+2)(hf'f^{-1})'] \e_2 \wedge dr
    \end{split}
\end{equation*}

Now using the formula 
$$
\D_S(\chi(\theta) \e_0)= (\D_S\chi(\theta)) -2\nabla_{\nabla \chi(\theta)}\e_0+ \chi(\theta)(\D_S \e_0)
$$
and the fact $d_S(\chi(\theta) \e_0)=(d_S \chi(\theta)) \wedge \e_0$ we get 
\begin{equation*}
    \begin{split}
     \D \omega=& hf^{-2}(\D_S \chi (\theta)) \e_0 \wedge dr +h f^{-2} (-2\nabla_{\nabla \chi(\theta)} \e_0) \wedge dr\\
     +&hf^{-2} \chi(\theta) \la_0 \e_0 \wedge dr +(-1)^k 2 h f^{-1}f'd_S(\chi(\theta)) \wedge \e_0\\
     -&[h''+(N-2k+2)(hf'f^{-1})'] \chi(\theta) \e_0 \wedge dr.\\
    \end{split}
\end{equation*}
In other words,
\begin{equation}\label{lastone}
    \begin{split}
     \D \omega=&\D_2(\p f^\m) \chi(\theta) \e_0\wedge dr\\
    +& \p f^{\m-2}(\D_S \chi (\theta)) \e_0 \wedge dr -2\p f^{\m-2} (\nabla_{\nabla \chi(\theta)} \e_0) \wedge dr\\
     +& (-1)^k \p f^\m (f^{-1}f')d_S(\chi(\theta)) \wedge \e_0\\
=&\D_2(\p f^\m) \chi(\theta) \e_0\wedge dr\\
    +& A_1+A_2+A_3.\\
    \end{split}
\end{equation}
We will choose $\p_\ep$ as in the proof of Theorem \ref{thm1}, but $\chi(\theta)$ will be the same for every $\ep>0$. Here $\m, \; \la$ we take the values $\m=-\frac{N}{p}+(k-1) +is$ for $s \in \mathbb R$ and $\la= -\m(\m+N-2k+2)$. As a result 

\begin{equation*}
    \begin{split}
       ||\D \omega - \la \omega ||^q_q \leq& ||(\D_2(\p f^\m) \e_0 \wedge dr - \la \p f^\m \e_0 \wedge dr) \chi(\theta)||^q_q\\
     +&  ||A_1||^q_q +||A_2||^q_q +||A_3||^q_q.
    \end{split}
\end{equation*}
Since $\chi(\theta)$ is a bounded function on $S^N$ with
$$
C_2 \leq \int_{S^N} |\e_0|_{S^N} \chi(\theta) \leq C_1
$$
the first term is estimated exactly as in the proof of Theorem \ref{thm1} by finding the appropriate $A_\ep, B_\ep$ for the domain of $\p_\ep$. So, it remains to bound the last three terms. Now by
$$\int_{S^N}|(\D_S \chi (\theta)) \e_0 \wedge dr|_M=C \int_{S^N}| \e_0|_S f^{-(k-1)},$$
$$\int_{S^N}|(\nabla_{\nabla \chi(\theta)} \e_0) \wedge dr|_M=C \int_{S^N} | \e_0|_S f^{-(k-1)},$$
and due to the additional factor of $f^{\m-2}$ in front of them in (\ref{lastone}), the estimates for $A_1$, $A_2$ are of the same type as $V$ in Theorem \ref{thm1}. 
To estimate $A_3$ we observe
$$
\int_{S^N}|d_S(\chi(\theta)) \wedge \e_0|_M=C| \e_0|_S f^{-k}.$$
As a result $A_3$ is also similar with $V$ but with an additional factor of $f^{\m-1}$ instead of $f^{\m-2}$ in  (\ref{lastone}), which still allows us to make it as small as we want by sending the support of $\p_\ep$ to $\infty$.

So, we have shown that the points
$$
\la=-\left(-\frac{N}{p}+ (k-1)+is\right)\left(-\frac{N}{p}+(N+1-k)+is\right), \;s \in \mathbb{R}
$$
belong to $\sigma (p,k, \Delta)$. Setting $k=N+1-m$, for $0\leq m\leq (N+1)/2$, in the above equation and changing sign in both brackets we get

$$
\la=-\left[N\left(\frac{1}{p}-1\right)+m-is\right]\left[\frac{N}{p}-m-is\right], \; s \in \mathbb{R},
$$
which are exactly the points of $P_{p,m}$.
Thus, for $0\leq m\leq (N+1)/2$ we have shown that $P_{p,m} \subset \sigma(p,N+1-m,\D)=\sigma(p,m,\D)$, where the equality follows by Poincare duality. 

\end{proof}
Finally, we are ready to give the proof of Theorem \ref{Theorem5.2}.

\textit{Proof of Theorem \ref{Theorem5.2}}.
If $1\leq p \leq 2$, then Lemma \ref{lemma6} gives 
 $$
    \sigma(k,p,\D_{\G,p}) \subset \{E_0,\dots,E_m\} \cup Q'_{p,k}.
    $$ and 
     $$
  \{E_0,\dots,E_m\} \subset  \sigma(k,p,\D_{\G,p})  .
    $$
    Now Lemma \ref{Lemma7} gives
$$
    Q'_{p,k}\subset \sigma(k,p,\D_{\G,p})
    $$ and the proof of Theorem \ref{Theorem5.2} follows.

\section{Appendix}

\subsection{Curvatures in warped products}
In this section we will describe the curvature of the manifold  $M$ when it is of the form $M=\mathbb{R} \times N$, where $(N,g_N)$ is an $(n-1)$  dimensional compact, complete smooth manifold and $M$ is  endowed with the warped product metric
$$
g_M=dr^2+f^2(r)g_N.
$$
Here $f(r)$ is a smooth function in $r$, called the warping function. 

One expects that the curvatures of $M$ will depend on the warping function $f(r)$ as well the curvature of $N$.  This can be established and we provide the exact formulas. 

Note that we will also be considering warped products in a slightly more general form of the type $(a,b) \times N$ with $(a,b) \subset \mathbb{R}$, but the curvature formulas will be the same.
 In (\cite{li2012geometric}, Appendix A) Li computes the sectional curvatures of $M$ using Cartan's structural equations (for these equations see for example  \cite{boothby1986introduction} p.380). Let $e_1=\frac{\partial}{\partial r}$ and $\widetilde{e_a}$, for $a \in (2, \dots, n)$ be an orthonormal frame on $N$. Then setting $e_a=\frac{1}{f}\widetilde{e_a}$  for $a \in (2, \dots, n)$ we have that $e_a$ for $a \in (1, \dots, n)$ is an orthonormal frame on $M$. Then, the sectional curvature of the planes $\pi_{1,\alpha}$
spanned by $e_1$ and $e_\alpha$ is given by 
\begin{equation} \label{eq_sec1}
  sec(\pi_{1,\alpha})=-\left((logf)''+((logf)')^2\right)=-\frac{f''}{f},  
\end{equation}
and the sectional curvature of the planes $\pi_{\alpha,\beta}$ spanned by $e_\alpha$ and $e_\beta$ is given by

\begin{equation} \label{eq_sec2}
   sec(\pi_{\alpha,\beta})=\frac{\widetilde{sec}(\pi_{\alpha,\beta})}{f^2}-((logf)')^2=\frac{\widetilde{sec}(\pi_{\alpha,\beta})-(f')^2}{f^2}, 
\end{equation}
where $\widetilde{sec}(\pi_{\alpha,\beta})$ denote the sectional curvature of the corresponding plane in $N$.

Now, we will show that all sectional curvatures of $M$ take values between 
$$
-\frac{f''}{f}, \;\;\frac{\widetilde{sec}(\pi_{\alpha,\beta})-(f')^2}{f^2}.
$$
Let us recall that the curvature operator 
$$
\mathcal{R} : \Lambda^2(TM) \to \Lambda^2(TM),
$$
is the self-adjoint operator uniquely defined by the relation
$$
g_M(\mathcal{R}(W\wedge X, Y\wedge Z)=-R(W,X,Y,Z).
$$
We have

\begin{prop}[\cite{petersen2006riemannian}, Prop. 4.1.1] \label{Proposition2.6}
    Let $e_i$ be an orthonormal basis for $T_pM$. If $e_i\wedge e_j$ diagonalize the curvature operator 
    $$
    \mathcal{R}(e_i \wedge e_j)=\la_{ij} e_i\wedge e_j,
    $$
    then for any plane $\pi$ in $T_pM$ we have $sec(\pi) \in [min\la_{ij},max \la_{ij}]$.
\end{prop}

So, since (\ref{eq_sec1}), (\ref{eq_sec2}) tell us that 

$$
\mathcal{R}(e_1\wedge e_\alpha)=-\frac{f''}{f}e_1\wedge e_\alpha,
$$
$$
\mathcal{R}(e_\alpha \wedge e_\beta)=\frac{\widetilde{sec}(\pi_{\alpha,\beta})-(f')^2}{f^2} e_\alpha \wedge e_\beta,
$$
we have the following

\begin{prop} \label{prop compact}
   Let $M=\mathbb{R} \times N$ be the product manifold endowed with the warped product metric $g_M=dr^2+f^2(r)g_N$, where $g_N$ is a metric on $N$. Then all  sectional curvatures of $M$ are between 
$$
-\frac{f''}{f}, \;\;\frac{\widetilde{sec}(\pi_{\alpha,\beta})-(f')^2}{f^2}.
$$
\end{prop}

\subsection{Heat kernel on forms}
 Now we will define the Weitzenbock tensor, a curvature tensor which will give us a way to find upper bounds for the heat kernel on $k$-forms.
 \begin{Definition}
     [Definition 1 \cite{charalambous2005lp}] Let $(M,g)$ be an oriented  Riemannian manifold and $V_i$ be a locally frame field and $\omega^i$ be its dual coframe field. We denote the tensor $W^k= -\sum_{i,j} \omega^i \wedge i(V_j) R_{V_iV_j}$  acting on $k$-forms, as the Weitzenbock tensor on $k$-forms, where $R_{XY}=D_XD_Y-D_YD_X -D_{[X,Y]}$ is the curvature tensor.
 \end{Definition}
For $1$-forms $W^1$ coincides with Ricci curvature, this is not true for higher order $k$.
\begin{prop}\label{Weitzenbock tensor} [ \cite{charalambous2005lp} Theorem 4,
  \cite{MR0458243}]
Let $M$ be a complete manifold with Ricci curvature bounded below and Weitzenbock tensor on $k$ forms bounded bellow $W^k\geq -K_2$. Then, the heat kernel on $k$ forms $\vec{p} (t,x,y)$ has the following pointwise bound 
    $$
   |\vec{p} (t,x,y)| \leq e^{K_2t} p(t,x,y),
    $$
where p(t,x,y) is the heat kernel of the Laplacian on functions.
\end{prop}
\begin{Corollary}
    \label{uniqueness of Cauchy on forms}

Let $M$ non-compact, complete, with Ricci curvature bounded from below and Weitzenbock tensor on $k$ forms bounded from below. Then, 

$$
\omega(x,t)= \int_M \vec{p}(t,x,y) \wedge * \omega_0(y) dy 
$$
is the unique solution to the heat equation on forms, 

$$
    \begin{cases}
     ( \partial_t +\Delta^k) \omega=0 \\
     \;\omega(\cdot, 0)=\omega_0
    \end{cases}
$$
\end{Corollary}

Finally, we write the following Propositions that we will need throughout this article. The first one is a generalization of the invariance of the integral under the isometry group, that is, if $J$ is an isometry, then
$$
\int_M f(J)dV_g=\int_M fdV_g
$$
 (see \cite{grigoryan2009heat} Lemma 3.27). 
 \begin{prop}\label{heat invariance}
    Let $\omega_1, \omega_2$ be $k$-forms and $J$ be an isometry. Then $$
   \int_M< \omega_1(Jx),\omega_2(Jx)>dx=\int_M< \omega_1(x),\omega_2(x)>dx.
    $$
\end{prop}
The second one is the invariance of the heat kernel on $k$-forms under the isometry group, whose proof follows from the previous proposition, and Corollary \ref{uniqueness of Cauchy on forms}. For the case of heat kernel on functions this is a standard fact, see for example \cite{grigoryan2009heat} Theorem 9.12) .
\begin{prop}\label{heat isometry}
    Let $J$ be an isometry. Then the heat kernel on forms satisfies $$
    \vec{p}(t,x,y)=\vec{p}(t,Jx,Jy).
    $$
\end{prop}

\bibliographystyle{alpha}
\bibliography{sample}

\end{document}